\documentclass[10pt, leqno]{article}

\usepackage{amsmath,amssymb,amsthm, epsfig}

\usepackage{lineno}

\usepackage{amsmath}

\usepackage[pdftex]{color}

\usepackage{comment}

\usepackage{mathtools}
\usepackage{color}

\usepackage{ulem}

\usepackage{mathrsfs}
\usepackage[mathscr]{euscript}

\usepackage{wasysym}

\usepackage{stackrel}

\usepackage{amsmath,amssymb,amsthm, epsfig}
\usepackage[colorlinks=true,urlcolor=blue,
citecolor=red,linkcolor=blue,linktocpage,pdfpagelabels,
bookmarksnumbered,bookmarksopen]{hyperref}
\usepackage{ulem}
\usepackage{dsfont}
\usepackage{mathrsfs}
\usepackage{mathtools}
\usepackage{stackrel}
\usepackage{esint,enumerate}

\usepackage[left=2.5cm,right=2.5cm,top=2.5cm,bottom=2.5cm]{geometry}

\def \dist {\mathrm{dist}}

\newcommand{\defeq}{\mathrel{\mathop:}=}

\newtheorem{theorem}{Theorem}[section]

\newtheorem{lemma}[theorem]{Lemma}

\newtheorem{proposition}[theorem]{Proposition}

\theoremstyle{definition}

\newtheorem{definition}[theorem]{Definition}

\theoremstyle{remark}

\numberwithin{equation}{section}

\usepackage{pgfplots}
\pgfplotsset{compat=newest}
\title{Schauder estimates for a class of fully nonlinear elliptic PDEs: a geometric tangential approach}
\author{\it by \smallskip \\ Junior da Silva Bessa \footnote{\noindent Universidade Estadual de Campinas - UNICAMP. Instituto de Matem\'{a}tica, Estat\'{i}stica e Computa\c{c}\~{a}o Cient\'{i}fica - IMECC. Departamento  de Matemática. Bar\~{a}o Geraldo, Campinas - SP, Brazil. \noindent \texttt{E-mail address: \url{jbessa@unicamp.br}}}, \quad Jo\~{a}o Vitor  da Silva
\footnote{\noindent Universidade Estadual de Campinas - UNICAMP Instituto de Matem\'{a}tica, Estat\'{i}stica e Computa\c{c}\~{a}o Cient\'{i}fica - IMECC. Departamento  de Matemática. Bar\~{a}o Geraldo, Campinas - SP, Brazil. \noindent \texttt{E-mail address: \url{jdasilva@unicamp.br}}},\\ \quad $\&$ \\\quad Laura Ospina\footnote{\noindent Universidade Estadual de Campinas - UNICAMP. Instituto de Matem\'{a}tica, Estat\'{i}stica e Computa\c{c}\~{a}o Cient\'{i}fica - IMECC. Departamento  de Matemática. Bar\~{a}o Geraldo, Campinas - SP, Brazil. \noindent \texttt{E-mail address: \url{l202049@dac.unicamp.br}}}
}
\begin{document}

\maketitle

\begin{abstract}
\noindent  In this paper, we establish local and global Schauder estimates for classical solutions of a class of fully nonlinear elliptic partial differential equations, which are not necessarily convex or concave, under suitable H\"{o}lder continuity assumptions on the data. Our approach is based on a robust blow-up argument, combined with geometric tangential analysis and compactness techniques. This strategy is strongly influenced by methodologies developed in the contemporary theory of nonlinear elliptic PDEs.

\medskip
\noindent \textbf{Keywords}: Fully nonlinear PDEs, local/global Schauder estimates, Blow-up technique, geometric tangential analysis.
\vspace{0.2cm}
	
\noindent \textbf{AMS Subject Classification: Primary 35B65; 35J15; 35J60; Secondary 35B53; 35B45.  
}

\end{abstract}

\section{Introduction}

%%%%%%%%%%%%%%%%%%%%%%%%%%%%%%%%%%%%%%%%%%%%%%%%%%%%%%%%%%%%%%%%%%%%%

The development of Schauder theory has significantly shaped the modern perspective that solving a PDE is, in essence, equivalent to establishing an \textbf{a priori} estimate, i.e., obtaining bounds on a solution before constructing it explicitly.

Beyond existence results (as the compactness properties they encode allow for the application of fixed-point theorems for compact operators), Schauder theory has numerous further applications. These include the analysis of asymptotic behavior—both at infinity and near singularities—as well as the study of qualitative properties of eigenfunctions, such as those arising in Riesz–Fredholm theory~\cite{Brezis1983} and the Krein–Rutman theorem~\cite{KreinRutman1950}. Additionally, Schauder theory plays a pivotal role in the method of sub- and super-solutions for nonlinear problems and in bifurcation theory (see \cite{Nirenberg2001} and \cite{Sattinger1971}). We refer the reader to the survey by Kichenassamy \cite{Kichenassamy2006} for a comprehensive account of Schauder estimates and their applications.

Throughout the last century, several alternative proofs were developed, among which we highlight the work of Campanato \cite{Campanato1964}, who introduced the so-called Campanato spaces; Peetre \cite{Peetre1966}, who employed the convolution of functions; Trudinger \cite{Trud1986}, who utilized mollification techniques,  and Simon \cite{Sim97} by applying a blow-up argument. 
Furthermore, perturbative approaches were developed by Safonov in \cite{Safonov1984} and \cite{Safonov1989}, and Caffarelli in \cite{Caffarelli1989} (see also \cite{CafCabre1995}) for fully nonlinear uniformly elliptic equations, methods that are also applicable to linear equations. More recently, Wang in \cite{Wang2006} proposed a novel and flexible method for establishing Schauder estimates for linear elliptic equations in non-divergence form, which combines an approximating scheme with compactness arguments. The above list of contributions on this topic is far too extensive and incomplete. In Section \ref{State-of-the-Art}, we will present other findings in the setting of fully nonlinear elliptic PDEs.

In the sequel, we present Schauder estimates for linear operators in non-divergence form.

\begin{theorem}[{\bf Interior estimates}]\label{thm:schauder}
Let $\alpha \in (0,1)$ fixed, and $u \in C^{2,\alpha}(\mathrm{B}_1)$ be any solution to
\[
-\mathrm{tr}(\mathfrak{A}(x)D^2u(x)) \defeq -\sum_{i,j=1}^{n} a_{ij}(x) \, \partial_{ij} u(x) = f(x) \quad \text{in } \mathrm{B}_1,
\]
with $f \in C^{0,\alpha}(\mathrm{B}_1)$ and $a_{ij} \in C^{0,\alpha}(\mathrm{B}_1)$, where the matrix $\mathfrak{A}(x) = (a_{ij}(x))_{i,j}$ satisfies the uniform ellipticity condition
\begin{equation}\label{eq:ellipticity}
\lambda|\xi|^2 \leq \sum_{i,j=1}^{n} a_{ij}(x)\xi_i\xi_j \leq \Lambda|\xi|^2 \quad \text{for all } x \in\mathrm{B}_1\,\,\,\text{and}\,\,\, \xi \in \mathbb{R}^n
\end{equation}
for some constants $0 < \lambda \leq \Lambda < \infty$. Then,
\[
\|u\|_{C^{2,\alpha}(\mathrm{B}_{1/2})} \leq \mathrm{C}\left( \|u\|_{L^{\infty}(\mathrm{B}_1)} + \|f\|_{C^{0,\alpha}(\mathrm{B}_1)} \right)
\]
for some constant $\mathrm{C} > 0$ depending only on $\alpha$, $n$, $\lambda$, $\Lambda$, and $\|a_{ij}\|_{C^{0,\alpha}(\mathrm{B}_1)}$.
\end{theorem}

Next, we present the counterpart of Theorem \ref{thm:schauder} in the context of global estimates (up to the boundary) for sufficiently smooth domains, and assuming the appropriate regularity of the coefficients and boundary datum.

\begin{theorem}[{\bf Boundary regularity}]\label{thm:boundary_regularity}
%Let $\alpha \in (0, 1)$ and $k \in \mathbb{N}$ with $k \geq 2$, and 
Let $\Omega \subset \mathbb{R}^n$ be a bounded domain of class $C^{2,\alpha}$. Let $u \in C^{2,  \alpha}(\Omega)$ be a solution to
\begin{equation}\label{eq:poisson_dirichlet}
\left\{
\begin{array}{rclcl}
-\displaystyle \sum_{i,j=1}^{n} a_{ij}(x) \, \partial_{ij} u(x) & = & f(x) & \text{in } & \Omega, \\
u(x) & = & g(x) & \text{on } & \partial\Omega,
\end{array}
\right.
\end{equation}
for $f \in C^{0,\alpha}(\Omega)$, $a_{ij} \in C^{0,\alpha}(\Omega)$ fulfilling the
ellipticity conditions for some $0< \lambda\leq \Lambda$, and $g \in C^{2,\alpha}(\partial\Omega)$. Then, the following estimate holds
\[
\|u\|_{C^{2,\alpha}(\overline{\Omega})} \leq \mathrm{C}\left( \|u\|_{L^{\infty}(\Omega)} + \|f\|_{C^{0,\alpha}(\Omega)} + \|g\|_{C^{2,\alpha}(\partial\Omega)} \right),
\]
for some constant $\mathrm{C}>0$ depending only on $\alpha$, $n$, $\lambda$, $\Lambda$, $\|\mathfrak{a}_{ij}\|_{C^{0, \alpha}(\Omega)}$ and $\Omega$.
\end{theorem}

\bigskip

Therefore, motivated by the previous introduction, this article has a twofold objective. First, we establish local \textit{a priori} Schauder estimates for classical solutions to the equation
$$
F(D^{2}u,x) +  \langle \mathfrak{B}(x), Du\rangle= f(x) \quad \text{in} \quad \mathrm{B}_1 \subset \mathbb{R}^n,
$$
under suitable assumptions on the data, which will be specified shortly. Here, $F:\text{Sym}(n)\times \mathrm{B}_1 \to \mathbb{R}$ denotes a fully nonlinear second-order operator, potentially including nonlinearities that lack a convexity/concavity structure. Furthermore, in the final part of the paper, we derive global estimates for classical solutions to the corresponding Dirichlet problem (with a flattened boundary),
$$
\left\{
\begin{array}{rcl}
F(D^2u,x) + \langle \mathfrak{B}(x), Du\rangle&=& f(x) \quad \text{in} \quad \mathrm{B}^{+}_{1}, \\
u(x) &=& g(x) \quad \text{on} \quad \mathrm{T}_{1},
\end{array}
\right.
$$
where $\mathrm{B}^{+}_1 \coloneqq \{x = (x^{\prime}, x_n) \in \mathbb{R}^n : x_n > 0 \,\, \text{and} \,\, |x| < 1\}$ and $\mathrm{T}_1 \coloneqq \{x = (x^{\prime}, x_n) \in \mathbb{R}^n : x_n = 0 \,\, \text{and} \,\, |x| < 1\}$.

\medskip

Throughout this manuscript, we assume the following structural conditions (cf. \cite{dosPrazTei2016}):
\begin{itemize}
\item[{\bf ($\mathrm{A}1$)}] ({\bf Uniform Ellipticity}) The operator $F:\text{Sym}(n)\times \Omega \to \mathbb{R}$ is fully nonlinear and uniformly elliptic, with ellipticity constants $0 < \lambda \leq \Lambda$. Specifically, we require
\begin{equation}\label{Unif_Elip}
\mathcal{M}^{-}_{\lambda, \Lambda}(\mathrm{N}) \leq F(\mathrm{M} + \mathrm{N},x) - F(\mathrm{M},x) \leq \mathcal{M}^{+}_{\lambda,\Lambda}(\mathrm{N}),
\end{equation}
for every $x \in \Omega$ and for all $\mathrm{M}, \mathrm{N} \in \text{Sym}(n)$ with $\mathrm{N} \geq 0$ (in the matrix sense), where
$$
\mathcal{M}_{\lambda, \Lambda}^+(\mathrm{M}) \coloneqq \sup_{\mathbf{A} \in \mathfrak{A}_{\lambda, \Lambda}} \operatorname{tr}(\mathbf{A} \mathrm{M}) \quad \text{and} \quad 
\mathcal{M}_{\lambda, \Lambda}^-(\mathrm{M}) \coloneqq \inf_{\mathbf{A} \in \mathfrak{A}_{\lambda, \Lambda}} \operatorname{tr}(\mathbf{A} \mathrm{M})
$$
denote the Pucci extremal operators, and
$$
\mathfrak{A}_{\lambda, \Lambda} \coloneqq \{\mathbf{A} \in \text{Sym}(n): \lambda \mathrm{Id}_n \leq \mathbf{A} \leq \Lambda \mathrm{Id}_n\}.
$$

\item[{\bf ($\mathrm{A}2$)}] ({\bf Differentiability of the Nonlinearity}) We assume that $F(\cdot,x)\in C^1(\text{Sym}(n))$ for all $x\in \Omega$, and there exists a modulus of continuity \(\omega: [0, \infty) \to [0, \infty)\) such that
$$
\|\mathrm{D}_{\mathrm{M}}F(\mathrm{X},x) - \mathrm{D}_{\mathrm{M}}F(\mathrm{Y},x)\| \leq \omega(\|\mathrm{X} - \mathrm{Y}\|),
$$
for all $x \in \Omega$ and for all \(\mathrm{X}, \mathrm{Y}, red\mathrm{M} \in \text{Sym}(n)\), 
where
$$
\mathrm{D}_{\mathrm{M}}F(\mathrm{X}_0,x) = \mathrm{D}F(\mathrm{X}_0,x)(\mathrm{M}) \coloneqq \lim_{t \to 0} \frac{F(t \mathrm{M} + \mathrm{X}_0,x) - F(\mathrm{X}_0,x)}{t}
$$
denotes the \textit{G\^{a}teaux derivative} of $F$ with respect to the first variable.

\item[{\bf ($\mathrm{A}3$)}] ({\bf Regularity of the Data}) There exists a modulus of continuity $\tau: [0, \infty) \to [0, \infty)$ such that \(\tau(t)\leq \mathrm{c}t^{\alpha}\) (for a dimensional constant $\mathrm{c}> 0$) and
$$
\left\{
\begin{array}{l}
|F(\mathrm{X},x) - F(\mathrm{X},y)| \leq \tau(|x - y|) \|\mathrm{X}\|_{\text{Sym}(n)}, \quad \forall\, x, y \in \Omega,\,\, \mathrm{X} \in \text{Sym}(n), \\
|f(x) - f(y)| \leq \tau(|x - y|), \quad \forall\, x, y \in \Omega, \\
\|\mathfrak{B}(x) - \mathfrak{B}(y)\| \leq \tau(|x - y|), \quad \forall\, x, y \in \Omega. \\
%\{x = (x^{\prime}, x_n) \in \mathbb{R}^n : x_n = 0\}.
\end{array}
\right.
$$
\end{itemize}

We denote by the following class the set of fully nonlinear operators satisfying the above assumptions:
$$
\mathscr{F} \coloneqq \{F: \text{Sym}(n)\times \Omega \to \mathbb{R} \,:\, F \text{ satisfies the assumptions}\,\,\,\, (\mathbf{A1})\text{--}(\mathbf{A3})\}.
$$

\bigskip

In this setting, we are able to derive the following local \textit{a priori} Schauder estimates.

\begin{theorem}[{\bf Interior Schauder estimates}]\label{T1}
Suppose the assumptions $(\mathrm{A1})-(\mathrm{A3})$ are in force with $\tau(t) \leq \mathrm{c} t^{\alpha}$ (for some \(\alpha\in (0,1)\)). Let \(u\in C^{2,\alpha}(\mathrm{B}_{1})\)  be a solution to
\[
F(D^{2}u,x)+\langle \mathfrak{B}(x),Du\rangle=f(x)\,\,\, \mathrm{in}\,\,\, \mathrm{B}_{1},
\]
where %\(F\in \mathscr{F}[(\lambda, \Lambda), \omega]\), \(f\in C^{0,\alpha}(\mathrm{B}_{1})\) and \
\(\|u\|_{C^{2,\alpha}(\mathrm{B}_{1})}\leq \mathrm{C}_{0}\) for some \(\mathrm{C}_{0}>0\). Then, there holds
\begin{eqnarray}\label{estinterior}
\|u\|_{C^{2,\alpha}(\mathrm{B}_{1/2})}\leq \mathrm{C}(n,\lambda,\Lambda,\alpha, \omega)\left(\|u\|_{L^{\infty}(\mathrm{B}_{1})}+\|f\|_{C^{0,\alpha}(\mathrm{B}_{1})}+\|\mathfrak{B}\|_{C^{0,\alpha}(\mathrm{B}_{1};\mathbb{R}^{n})}\right).  
\end{eqnarray}
\end{theorem}

Moreover, we also address the following global counterpart of the previous result.

\begin{theorem}[{\bf Global Schauder estimates}]\label{T2}
Suppose the assumptions $(\mathrm{A1})-(\mathrm{A3})$ are in force with $\tau(t) \leq \mathrm{c}t^{\alpha}$ (for some \(\alpha\in (0,1)\)). Let \(u\in C^{2,\alpha}(\mathrm{B}_{1}^{+} \cup \mathrm{T}_1)\) be a solution to
\[
\left\{
\begin{array}{rcl}
F(D^2u,x) + \langle \mathfrak{B}(x), Du\rangle &=& f(x) \quad \mbox{in} \,\,\,\,\,   \mathrm{B}^{+}_{1}\\
u(x)&=& g(x)  \quad \mbox{on} \,\,\,\,  \mathrm{T}_{1}
\end{array}
\right.
\]
where %\(F\in \mathscr{F}\), \(f\in C^{0,\alpha}(\mathrm{B}_{1})\), \(g\in C^{2,\alpha}(\mathrm{T}_{1})\) and 
\(\|u\|_{C^{2,\alpha}(\overline{\mathrm{B}_{1}^{+}})}\leq \mathrm{C}^{\ast}_{0}\) for some \(\mathrm{C}^{\ast}_{0}>0\) and $g$ satisfies $|g(z) - g(w)| \leq \tau(|z - w|), \ \forall\, z, w \in \mathrm{T}_{1}$. Then, there holds
\begin{equation}\label{estboundary}
\|u\|_{C^{2,\alpha}(\mathrm{B}_{1/2}^{+})}\leq \mathrm{C}\left(\|u\|_{L^{\infty}(\mathrm{B}_{1}^{+})}+\|f\|_{C^{0,\alpha}(\mathrm{B}_{1}^{+})}+\|g\|_{C^{2,\alpha}(\mathrm{T}_{1})}+\|\mathfrak{B}\|_{C^{0,\alpha}(\mathrm{B}_{1};\mathbb{R}^{n})}\right),  
\end{equation}
where \(\mathrm{C}=\mathrm{C}(n,\lambda,\Lambda,\alpha, \omega)>0\) is a universal constant.
\end{theorem}

To establish our Schauder estimates, we employ a robust blow-up technique that traces back to Simon's seminal work~\cite{Sim97} (cf. \cite{FR-O} for a modern survey developing this technique), combined with geometric tangential analysis and compactness arguments (inspired by methods from the modern theory of elliptic partial differential equations, cf.~\cite{Teixeira2016} and \cite{Teixeira2020} for insightful surveys on this research area), and classification of global profiles (Liouville type-results). As discussed previously, Schauder estimates, both local and global versions have been extensively studied over the past decades using a variety of analytical tools (see the State-of-the-Art in Section~\ref{State-of-the-Art} for more details). Nevertheless, to the best of our knowledge, such estimates have not previously been established in the fully nonlinear setting (with a drift term and inhomogeneous Dirichlet boundary conditions) using blow-up arguments in conjunction with geometric tangential methods. This gap in the literature was one of the principal motivations behind the development of the estimates presented in this work.

In a summarized way, the main contributions of the paper are as follows:

\begin{itemize}
\item[\textbf{(i)}] \textbf{Schauder estimates beyond the convex/concave setting:}
We establish local and global $C^{2,\alpha}$ Schauder estimates for classical solutions of uniformly elliptic fully nonlinear equations without imposing convexity or concavity assumptions on the operator. Thus, our framework applies to nonlinear operators beyond the standard convexity/concavity setting (cf. \cite{BhatWarr2021,CabreCaff2003,Goffi2024}).

    \item[\textbf{(ii)}] \textbf{Spatial dependence and lower-order terms:}
    Our results apply to equations of the form
    \[
    F(D^2u,x)+\langle \mathfrak{B}(x),Du\rangle=f(x),
    \]
    allowing for both spatial dependence of the fully nonlinear operator and a nontrivial drift term. The operator, drift term, and right-hand side are assumed to satisfy suitable H\"{o}lder continuity conditions (cf. \cite{daSNor21,dosPrazTei2016}).

    \item[\textbf{(iii)}] \textbf{Interior and global estimates:}
    We establish both interior Schauder estimates and global estimates for the corresponding Dirichlet problem in a flattened boundary setting. The latter accommodates inhomogeneous Dirichlet data and yields boundary regularity within the same fully nonlinear framework (cf. \cite{daSNor21,dosPrazTei2016,Teixeira2016}).

    \item[\textbf{(iv)}] \textbf{Geometric tangential blow-up method:}
    The proofs rely on a geometric tangential blow-up argument combined with compactness, stability, and Liouville-type classification of global profiles. Under the natural scaling, the nonlinear operator converges to a linear uniformly elliptic tangential operator, while the rescaled lower-order terms and inhomogeneity vanish. This reduction enables us to exploit the rigidity of the limiting profiles and recover the desired $C^{2,\alpha}$ estimates for the original nonlinear equation (cf. \cite{FR-O,Sim97,Teixeira2016}).

    \item[\textbf{(v)}] \textbf{A unified perturbative framework:}
    Our argument provides a flexible mechanism for transferring regularity from the limiting linear tangential equation to the original nonlinear problem. In particular, it does not rely on convexity/concavity-based methods and is not restricted to situations in which the nonlinear equation can be treated directly through a conventional linearization procedure (cf. \cite{BhatWarr2021,BdaSO26,dosPrazTei2016}).
\end{itemize}

To the best of our knowledge, this combination of structural assumptions, estimates, and proof techniques has not previously been developed in this generality. In particular, we are not aware of previous works establishing both interior and global Schauder estimates for fully nonlinear uniformly elliptic equations with spatially dependent nonlinearities and drift terms, without convexity or concavity assumptions, through a geometric tangential blow-up approach.

\subsection*{Explaining the heuristic behind the proof}

We conclude this introduction by outlining the heuristic framework of the geometric tangential analysis underlying our proofs. By way of explanation, consider a fully nonlinear elliptic operator \( F: \text{Sym}(n) \to \mathbb{R} \) such that $F(\mathbf{O}_n) = 0$ (this is not a restrictive assumption). Then, the family of elliptic scaling functions defined by
\[
\mathcal{G}_{\varrho}(\mathrm{X}) := \frac{1}{\varrho} F(\varrho \mathrm{X}), \quad \text{for } \varrho > 0.
\]
constitutes a continuous family of operators that preserves the ellipticity constants of the original equation (i.e., $0< \lambda\leq \Lambda$). Indeed, $\mathscr{F}$ invariant under the transformation $\varrho \rightarrow \frac{1}{\varrho}F(\varrho x)$. 

Now, if \( F: \text{Sym}(n) \to \mathbb{R}\) is differentiable (for instance, at the origin, recalling the normalization \( F(\mathbf{O}_n) = 0 \)), then we have
\[
\displaystyle \lim_{\varrho \to 0} \mathcal{G}_{\varrho}(\mathrm{X}) = \lim_{\varrho \to 0} \frac{F(\varrho \mathrm{X}+\mathbf{O}_n)-F(\mathbf{O}_n)}{\varrho} = \mathrm{D}_{\mathrm{X}}F(\mathbf{O}_n) = \sum_{i,j=1}^{n}\frac{ \partial F}{\partial \mathrm{X}_{ij}}(\mathbf{O}_n) \mathrm{X}_{ij}.
\]
In other words, the linear, uniformly elliptic operator \(\displaystyle \mathrm{X} \mapsto \sum_{i,j=1}^{n}\frac{ \partial F}{\partial \mathrm{X}_{ij}}(\mathbf{O}_n) \mathrm{X}_{ij}  = \mathrm{tr}(\mathfrak{A}_0\mathrm{X})\) represents the tangential equation associated with \(\mathcal{G}_{\varrho} \) in the limiting configuration as \( \varrho \to 0 \). 

Now, if \( u: \mathrm{B}_1 \to \mathbb{R} \) is a solution to an equation involving the original operator \( F: \text{Sym}(n) \to \mathbb{R}\), i.e.,
$$
F(D^2 u) + \langle \mathfrak{B}(x), Du\rangle = f(x) \quad \text{in} \quad \mathrm{B}_1,
$$
then the scaled function \( u^{\mathfrak{a}}_{\varrho}(x) := \frac{1}{\varrho^{2+\alpha}\mathfrak{a}} u(\varrho x) \) (a blow-up) satisfies a corresponding equation for 
$$
\mathcal{G}_{\varrho^{\alpha} \mathfrak{a}}(D^2 u^{\mathfrak{a}}_{\varrho}(x)) + \langle \mathfrak{B}_{\varrho}^{\mathfrak{a}}(x), D u^{\mathfrak{a}}_{\varrho}(x)\rangle = \frac{1}{\varrho^{\alpha} \mathfrak{a}} f(\varrho x) := f_{\varrho}^{\mathfrak{a}}(x) \quad \text{in} \quad \mathrm{B}_{1/\varrho}, 
$$
where $\mathfrak{B}_{\varrho}^{\mathfrak{a}}(x) = \frac{\varrho}{a} \mathfrak{B}(\varrho x)$, and  $\mathfrak{a}>0$ is chosen in such a way $\|f_{\varrho}^{\mathfrak{a}}\|_{\infty, \mathrm{B}_1} = \text{o}(1)$ and $\|\mathfrak{B}_{\varrho}^{\mathfrak{a}}\|_{\infty, B_1} = \text{o}(1)$ as $\varrho \to 0$.

Moreover, if it is further proved that the norm \( [D^{2}  u^{\mathfrak{a}}_{\varrho} ]_{\alpha,\mathrm{B}_{1/(2\varrho)}} \) is bounded, e.g., by \( 1 \), then \( u^{\mathfrak{a}}_{\varrho} \) becomes a normalized solution (in the $(2+\alpha)-$fashion) to the ``\( \varrho \)-scaled equation''. This strategy allows us to access some rigidity results (e.g., classification of global profiles) available for the linear tangential equation, via compactness and stability arguments.

Indeed, we can conclude the following:
$$
\begin{array}{ccc}
 \left\{ 
 \begin{array}{ccl}
 \mathcal{G}_{\varrho^{\alpha} \mathfrak{a}}(\mathrm{X}) \to  \mathrm{tr}(\mathfrak{A}_0 \mathrm{X}) & \text{as}& \varrho \to 0    \\
 \mathfrak{B}_{\varrho}^{\mathfrak{a}} \to 0 & \text{as} & \varrho \to 0\\
 f_{\varrho}^{\mathfrak{a}} \to 0 & \text{as}& \varrho \to 0\\
 u^{\mathfrak{a}}_{\varrho} \to \mathbf{U} & \text{as}& \varrho \to 0
 \end{array}
 \right.& \Rightarrow & \mathrm{tr}(\mathfrak{A}_0 D^2 \mathbf{U}) = 0 \quad \text{in} \quad \mathbb{R}^n \quad \text{and} \quad [D^{2}  \mathbf{U} ]_{\alpha,\mathbb{R}^n} \leq 1.
\end{array}
$$

Finally, since $0 < \lambda \mathrm{Id}_n \leq \mathfrak{A}_0 \leq  \Lambda \mathrm{Id}_n$, after changing variables (namely, $z = \mathfrak{A}_0^{1/2}x$), the case of constant coefficients (uniformly elliptic), one reduces to analyzing the scenario of harmonic functions, which has a rich available potential theory (for instance, concerning Liouville-type results)
$$
 \mathrm{tr}(\mathfrak{A}_0 D^2 \mathbf{U}) = 0 \quad \text{in} \quad \mathbb{R}^n \quad \Leftrightarrow \quad \sum_{i=1}^{n} \partial_{z_iz_i} \mathbf{U} = 0 \quad \text{in} \quad \mathbb{R}^n \quad (\text{with}\,\,\,z = \mathfrak{A}_0^{1/2}x)
$$
\medskip

For the inhomogeneous Dirichlet boundary case, a similar reasoning can be employed (provided we assume $\|g_{\varrho}^{\mathfrak{a}} = u^{\mathfrak{a}}_{\varrho}|_{\partial \mathrm{T}_{1}}\|_{\infty, \mathrm{B}_1} = \text{o}(1)$), yielding a profile defined in the half-space, which necessitates the application of a Liouville-type result in this context.

Therefore, this approach will enable us to apply \textit{Simon's blow-up technique}, developed in \cite{Sim97} (see also \cite{FR-O}), to obtain Schauder estimates in both local and global settings.

\section{A brief literature review on Schauder  estimates}\label{State-of-the-Art}

\subsection*{Higher estimates for $2$nd-order fully nonlinear elliptic PDEs: Local scenario}

In the mid-20th century, Nirenberg~\cite{Nirenberg1953} pioneered H\"older estimates for the Hessian of classical solutions to certain nonlinear partial differential equations in two dimensions.

\begin{theorem}[{\cite[Theorem 4.9]{FR-O}}]
Let \( F : \mathbb{R}^{2 \times 2} \to \mathbb{R} \) be uniformly elliptic with ellipticity constants \( \lambda \) and \( \Lambda \). Let \( u \in C^2(\mathrm{B}_1) \) be a solution of
\[
F(D^2u) = 0 \quad \text{in } \quad \mathrm{B}_1 \subset \mathbb{R}^2.
\]
Then,
\[
\|u\|_{C^{2,\beta}(\mathrm{B}_{1/2})} \leq \mathrm{C} \|u\|_{L^{\infty}(\mathrm{B}_1)},
\]
for some constants \( \beta > 0 \) and \( \mathrm{C} > 0 \), depending only on \( \lambda \) and \( \Lambda \).
\end{theorem}

In higher dimensions, a seminal result was independently obtained by Evans~\cite{Evans1982} and Krylov~\cite{Krylov1982} in the early 1980s, establishing the following:

\begin{theorem}[{\bf Evans--Krylov's Theorem}]\label{thm:evans-krylov}
Let \( F: \text{Sym}(n) \to \mathbb{R} \) be a convex (or concave) uniformly elliptic operator satisfying \( F(0) = 0 \). Let \( u \in C(\mathrm{B}_1) \) be a viscosity solution of
\[
F(D^2 u) = 0 \quad \text{in } \quad  \mathrm{B}_1.
\]
Then,
\[
\|u\|_{C^{2,\beta}(\mathrm{B}_{1/2})} \leq \mathrm{C} \|u\|_{L^{\infty}(\mathrm{B}_1)},
\]
for some \( \beta > 0 \) and constant \( \mathrm{C} > 0 \), depending only on the dimension \( n \), and the ellipticity constants \( \lambda \) and \( \Lambda \). Particular, if \( F \) is smooth, then \( u \in C^\infty(\mathrm{B}_1) \).
\end{theorem}

Subsequently, Caffarelli addressed Schauder estimates for the inhomogeneous problem by using perturbation and compactness techniques.

\begin{theorem}[{\bf Caffarelli's Schauder estimates - \cite[Ch. 8]{CafCabre1995}}] Let \( F: \text{Sym}(n) \to \mathbb{R} \) be a convex, uniformly elliptic operator with ellipticity constants \( 0< \lambda \le \Lambda \), satisfying \( F(\mathbf{O}_n) = 0 \). Suppose that \( u \in C^0(B_1) \) is a viscosity solution to
\[
F(D^2 u) = f(x) \quad \text{in } B_1 \subset \mathbb{R}^n.
\]
Then, there exists a constant \( \alpha_0 \in (0,1) \), depending only on \( n, \lambda \) and \( \Lambda \), such that if \( f \in C^{0, \alpha}(B_1) \) for some \( \alpha \in (0,1) \), then \( u \in C^{2,\min\{\alpha_0, \alpha\}}(B_{1/2}) \), and
    \[
    \|u\|_{C^{2,\min\{\alpha_0,\alpha\}}(B_{1/2})} \leq \mathrm{C}\left( \|u\|_{L^{\infty}(B_1)} + \|f\|_{C^{0, \alpha}(B_1)} \right),
    \]
    for some constant \( \mathrm{C} > 0 \) depending only on \( n, \lambda \), \( \Lambda \), and \( \alpha \). 
\end{theorem}

\medskip

Trudinger \cite{Trud1983} derived first- and second-order derivative estimates for classical solutions to fully nonlinear, uniformly elliptic equations of the form
\[
F(D^2u,Du,u,x) = 0 \quad \text{in} \quad \Omega \subset \mathbb{R}^n,
\]
under natural structural assumptions. These estimates were applied to the Dirichlet problem for Bellman equations, thereby generalizing earlier results by Lions \cite{Lions1981} and Evans \cite{Evans1983}. Furthermore, Evans \cite{Evans1982} originally presented a concise treatment of H\"older estimates for second derivatives.

\medskip

Huang \cite{Huang2002} subsequently investigated conditions under which a $C^{1,1}$ solution of the fully nonlinear, uniformly elliptic equation $F(D^2u) = 0$ must be classical (i.e., $C^2$). The work demonstrates that if $F \in C^1$ and satisfies the Liouville property, then every $C^{1,1}$ viscosity solution belongs to $C^{2,\alpha}$ for all $0 < \alpha < 1$. Recall that an operator $F$ possesses the \textbf{Liouville property} if every locally $C^{1,1}$ viscosity solution of $F(D^2u) = 0$ in $\mathbb{R}^n$ with bounded second derivatives is necessarily a quadratic polynomial.

\medskip

Cabr\'{e} and Caffarelli \cite{CabreCaff2003} examined intermediate conditions on $F$, weaker than convexity but stronger than no assumptions, that guarantee classical solutions to $F(D^2u) = 0$. Their work establishes local $C^{2,\alpha}$ regularity estimates for a class of non-convex/concave operators.

\medskip

We further note that Monneau \cite[Proposition 9.1]{Monneau2009} established pointwise $C^{2,\alpha}$ estimates (in the $L^p$ norm) for solutions to $F(D^2 u) = 0$ in $B_1$, requiring only $F \in C^2$ without convexity or concavity assumptions. These results hold under the condition of pointwise $C^2$-Dini regularity (in the $L^{\infty}$ norm).

\medskip

Cao \textit{et al.} \cite{CLW2011} obtained local $C^{2,\alpha}$ estimates for classical solutions of fully nonlinear, uniformly elliptic equations $F(D^2u) = 0$, expressed in terms of the Hessian matrix $D^2u$'s modulus of continuity. Their analysis assumes $F$ is merely locally $C^{1,\beta}$, dispensing with convexity or concavity requirements. Through an iterative scheme based on $L^2$ estimates, they derived H\"older continuity for the Hessian.

\medskip

For nearly twenty years, the question of whether \textit{arbitrary} fully nonlinear elliptic operators admit a general $C^2$ \textit{a priori} regularity theory remained unresolved. This problem was finally settled by Nadirashvili and Vl\u{a}du\c{t}'s counterexamples to $C^{1,1}$ regularity \cite{NV2007,NV2008}, which concluded this line of inquiry. Their work, however, stimulated new research directions. In light of the inherent obstacles to developing a universal existence theory for classical solutions to fully nonlinear equations, a major focus of current research involves identifying supplementary structural or qualitative conditions on both $F$ and $u$ that enable $C^2$ estimates.

\medskip

Regarding recent developments in local higher regularity estimates, Savin's seminal work \cite{Savin2007} examines viscosity solutions to general second-order fully nonlinear equations of the form
\[
F(D^2u, Du, u, x) = 0,
\]
where $u = 0$ constitutes a solution. Under the assumptions that $F$ is smooth and uniformly elliptic solely in a neighborhood of $(0, 0, 0, x)$, Savin proves interior $C^{2,\alpha}$ regularity for \textbf{flat solutions} - those with sufficiently small $L^{\infty}$ norm. Notably, the operator $F$ is initially taken to be merely measurable, with a Harnack inequality first established for flat solutions; higher regularity results then follow under additional smoothness assumptions on $F$.

\medskip

In their comprehensive study \cite{dosPrazTei2016} (see also \cite[Chapter 5]{João}), dos Prazeres and Teixeira investigate equations of the form
\begin{equation}\label{Eq-PrazTei}
F(D^2u,x) = f(x),
\end{equation}
where $F$ represents a non-convex, fully nonlinear, uniformly elliptic operator. The authors derive local $C^{2,\alpha}$ regularity for flat solutions to \eqref{Eq-PrazTei} (solutions with sufficiently small norm), provided the coefficients of $F$ and the source term $f$ belong to $C^{0,\alpha}$. 

\begin{theorem}[{\bf  $C^{2,\alpha}$ regularity - \cite[Theorem 2.2]{dosPrazTei2016}}] 
Let \( u \in C^0(\mathrm{B}_1) \) be a viscosity solution to
\[
F(D^2u,x) = f(x) \quad \text{in }  \quad \mathrm{B}_1,
\]
where \( F \) and \( f \) satisfy $(\mathrm{A1})–(\mathrm{A3})$ with \( \tau(t) = \mathrm{c}t^\alpha \) for some \( 0 < \alpha < 1 \). There exists \( \delta_0 > 0 \), depending only upon \( n, \lambda, \Lambda, \omega, \alpha \), and \( \tau(1) \), such that if
\[
\sup_{\mathrm{B}_1} |u| \leq \delta_0,
\]
then \( u \in C^{2,\alpha}(\mathrm{B}_{1/2}) \) and
\[
\|u\|_{C^{2,\alpha}(\mathrm{B}_{1/2})} \leq \mathrm{C}_0 \cdot \delta_0,
\]
where \( \mathrm{C}_0>0\) depends only upon \( n, \lambda, \Lambda, \omega \), and \( (1 - \alpha) \).
    
\end{theorem}

The proof of such a result is based on a combination of geometric tangential analysis and perturbative arguments inspired by compactness methods from the theory of elliptic partial differential equations. The authors construct a family of elliptic rescalings that preserve the ellipticity constants of the original operator, thereby allowing for the application of tangential linear elliptic regularity theory.

\medskip

In 2021, Bhattacharya and Warren, in \cite{BhatWarr2021}, established explicit interior \( C^{2,\alpha} \) estimates for viscosity solutions of fully nonlinear, uniformly elliptic equations that are sufficiently close to linear equations. Moreover, they provided a precise quantitative bound characterizing this closeness.

\begin{definition}[{\cite[
Definition 1.2]{BhatWarr2021}}] A uniformly elliptic, non-linear operator $F$ is said \textbf{almost linear} with constant $\varepsilon>0$ if
\begin{equation}
\|\mathrm{D}_{\mathrm{X}}F(\mathrm{M}) - \mathrm{D}_{\mathrm{X}}F(\mathrm{N})\| \leq \varepsilon
\end{equation}
for all $\mathrm{M}, \mathrm{N} \in \text{Sym}(n)$. Moreover, one defines $\varepsilon$ to be the \textbf{closeness constant} of $F$.
    
\end{definition}

\begin{theorem}[{\cite[
Theorem 1.4]{BhatWarr2021}}] Given $0<\lambda\leq \Lambda$, and $0 < \alpha < \overline{\alpha} < 1$, suppose that $F$ is almost linear with constant $\varepsilon_0$, and let $u \in C^0(\mathrm{B}_1)$ be a viscosity solution of 
$$
F(D^2u) = f(x) \quad \text{in} \quad \mathrm{B}_1. 
$$
If $f \in C^{0,\alpha}(\mathrm{B}_1)$, then $u \in C^{2,\alpha}(\mathrm{B}_{1/2})$ and the following estimate holds:
\begin{equation}
\|u\|_{C^{2,\alpha}(\mathrm{B}_{1/2})} \leq \mathrm{C}\left( \|u\|_{L^{\infty}(\mathrm{B}_1)} + \|f\|_{C^{0,\alpha}(\mathrm{B}_1)} \right),
\end{equation}
where $\mathrm{C}>0$ depends only on $n$, $\lambda$, $\Lambda$, $\varepsilon_0$, $\alpha$, and $\overline{\alpha}$.
\end{theorem}

\medskip 

Another condition that can lead to improved regularity arises when the ellipticity constants \( 0 < \lambda \leq \Lambda \) are sufficiently close. Specifically, Wu and Niu, in \cite[Theorem 1.4]{WuNiu2023}, employ compactness techniques to establish interior \( C^{2,\alpha} \) regularity for viscosity solutions of fully nonlinear, uniformly elliptic equations under the assumption that the ellipticity constants are nearly equal. It is worth noting that the homogeneous version of this result had already been obtained nearly a decade earlier by Da Silva in his Ph.D. thesis (see \cite[Ch. 5]{DaSilva-PhDThesis}).

\medskip

Finally, in 2024, Goffi, in \cite{Goffi2024}, undertakes a comprehensive study of \textit{a priori} estimates and Evans-Krylov-type regularity results in Hölder spaces for fully nonlinear, second-order, uniformly elliptic or uniformly parabolic equations, even in the absence of concavity or convexity of the operator on the space of symmetric matrices. As a consequence, the results yield Liouville-type theorems for polynomial solutions to elliptic equations, as well as for certain parabolic problems. Furthermore, Goffi addresses the following generalization of Nirenberg’s classical result \cite{Nirenberg1953} to the setting of viscosity solutions.

\begin{theorem}[{\bf \cite[Theorem Appendix A.1]{Goffi2024}}]
    
Let $u : \mathrm{B}_1 \to \mathbb{R}$, with $\mathrm{B}_1 \subset \mathbb{R}^2$, and suppose that $u$ is a continuous viscosity solution to
\[
F(D^2 u) = f(x) \quad \text{in } \mathrm{B}_1.
\]
Assume that $F : \text{Sym}(2) \to \mathbb{R}$ is uniformly elliptic (no other assumptions are required), and that $f \in C^{0,\alpha}(\mathrm{B}_1)$. Then, for some small $\alpha \in (0, 1)$, we have the regularity estimate
\[
\|u\|_{C^{2,\alpha}(\mathrm{B}_{1/2})} \leq \mathrm{C}\left( \|u\|_{L^{\infty}(\mathrm{B}_1)} + \|f\|_{C^{0, \alpha}(\mathrm{B}_1)} \right),
\]
where the constants $\alpha$ and $\mathrm{C}>0$ depend only on $\lambda$ and $\Lambda$.
\end{theorem}

\subsection*{Higher estimates for $2$nd-order fully nonlinear PDEs: up-to-boundary scenario}

Krylov, in \cite{Krylov1982}, studied operators of the form \(\displaystyle F[\cdot] = \inf_j F_j[\cdot] \) for \( j = 1, 2, \ldots \), where each \( F_j \) is positively homogeneous of degree one. Under appropriate structural assumptions, he established the existence and uniqueness of classical solutions \( u \in C^{2,\alpha} \) to fully nonlinear second-order elliptic boundary value problems:
\[
F(D^2 u, D u, u, x) = 0 \quad \text{in } \Omega, \qquad u = g \quad \text{on } \partial \Omega.
\]
The argument hinges on the derivation of a priori estimates in Hölder spaces. Notably, the Hamilton-Jacobi-Bellman and Monge-Ampère equations exemplify the applicability of this framework.

\medskip 
Subsequently, in \cite{Krylov1983}, Krylov extended his analysis to operators \( F \) belonging to the class of so-called boundedly inhomogeneous functions. This class includes, for instance, operators of the form \(\displaystyle F[\cdot] = \inf_k F_k[\cdot] \), where each \( F_k \) is again positively homogeneous of degree one. In this setting, Krylov proved boundary regularity results, showing that solutions belong to \( C^{2,\alpha}(\overline{\Omega}) \). His proof is technically intricate, involving a reduction to a degenerate boundary equation and the subsequent derivation of Hölder estimates for this auxiliary problem.

\medskip 
In a later development, Safonov, in \cite{Safonov1989}, investigated the Dirichlet problem for fully nonlinear second-order elliptic equations of the form
\begin{equation}\label{Eq-Safonov}
 F\left(u_{x_i x_j}, u_{x_i}, u, x\right) = 0 \quad \text{in } \Omega, \quad u = g \text{ on } \partial\Omega,
\end{equation}
where $\Omega \subset \mathbb{R}^n$ is an open domain. This general formulation encompasses classical problems such as the Bellman equation:
\[
\displaystyle  F\left(u_{x_i x_j}, u_{x_i}, u, x\right) \coloneqq \sup_k \left[ \sum_{i, j=1}^n a_{ij}^k(x) u_{x_i x_j} + \sum_{i=1}^n b_i^k(x) u_{x_i} + c^k(x) u + f^k(x) \right] = 0,
\]
with coefficients belonging to \( C^\alpha(\Omega) \) and uniformly bounded. Assuming appropriate conditions on \( F \), \( \Omega \), and the boundary data \( g \), Safonov established that the problem \eqref{Eq-Safonov} admits a unique solution in the space \( C^{2,\alpha}_{\text{loc}}(\Omega) \cap C^0(\Omega) \) for sufficiently small \( \alpha > 0 \). Moreover, if \( \partial\Omega \in C^{2,\alpha} \)—for example, when \( \Omega = \{ x \in \mathbb{R}^n : \psi(x) > 0 \} \) with \( \psi \in C^{2,\alpha}(\mathbb{R}^n) \) satisfying \( |D \psi| \geq 1 \) on \( \partial\Omega \)—and \( g \in C^{2,\alpha}(\Omega) \), then the solution actually belongs to \( C^{2,\alpha}(\Omega) \).

\medskip 
More recently, Silvestre and Sirakov, in \cite{SilSir2014}, analyzed the boundary regularity of continuous viscosity solutions to fully nonlinear, uniformly elliptic equations of the form
\[
F(D^2u, Du, x) = f(x)
\]
posed in bounded domains \( \Omega \subset \mathbb{R}^n \), with Dirichlet boundary conditions on a portion of \( \partial\Omega \). The same question is also addressed for viscosity solutions of the differential inequalities
\[
\mathcal{M}^+_{\lambda,\Lambda}(D^2u) + \mathfrak{L}|Du| \geq -\|f\|_{L^\infty(\Omega)} \quad \text{and} \quad 
\mathcal{M}^-_{\lambda,\Lambda}(D^2u) - \mathfrak{L}|Du| \leq \|f\|_{L^\infty(\Omega)},
\]
where \( \mathcal{M}^\pm_{\lambda,\Lambda} \) denote the Pucci extremal operators and \( \mathfrak{L} \geq 0 \) is the Lipschitz constant of \( F \) with respect to the gradient variable. It is crucial to emphasize that the operator \( F \) is not assumed to be convex or concave in \( D^2u \). To address the lack of structural regularity, the authors employ sup-convolutions as an essential technical tool. Assuming sufficient regularity for the boundary data \( g = u|_{\partial\Omega} \) and under suitable structural assumptions on \( F \), they show that for every continuous viscosity solution \( u \), there exist functions \( \mathscr{G} \in C^{0,\alpha}(\partial\Omega, \mathbb{R}^n) \) and \( \mathbb{H} \in C^{0,\alpha}(\partial\Omega, \text{Sym}(n)) \), representing the boundary gradient and Hessian of \( u \), respectively. More precisely, the following Taylor-type expansion holds:
\[
\left|u(x) - u(x_0) - \mathscr{G}(x_0) \cdot (x - x_0) - \frac{1}{2} \mathbb{H}(x_0)(x - x_0) \cdot (x - x_0)\right| \leq \mathrm{C} \mathrm{W} |x - x_0|^{\alpha + 2},
\]
for all \( x_0 \in \partial\Omega \) and \( x \in \Omega \cap B_r(x_0) \), where the constants \( \mathrm{C} \), \( \mathrm{W} \), and \( \alpha \) depend on the ellipticity constants \( \lambda \), \( \Lambda \), the dimension \( n \), the regularity of \( \Omega \), and the Hölder norms of the boundary data \( g = u|_{\partial\Omega} \), the source term \( f \), and \( \|u\|_{L^\infty(\Omega)} \).

\section{Preliminaries and auxiliary results}

In this section, we provide essential definitions and auxiliary results that are fundamental to our approach in establishing the Schauder estimates, within the context of our discussion.

First, we recall the definition of H\"{o}lder spaces.

\begin{definition}[{\bf H\"{o}lder Spaces} - {\cite[Pag. 6]{FR-O}}]
Given $\alpha \in (0, 1]$, the \textbf{H\"{o}lder space} $C^{0,\alpha}(\overline{\Omega })$ consists of all continuous functions $u \in C(\overline{\Omega })$ for which the H\"{o}lder semi-norm 
\[
[u]_{C^{0,\alpha}(\overline{\Omega } )} := \sup_{\substack{x, y \in \overline{\Omega }  \\ x \neq y}} \frac{|u(x) - u(y)|}{|x - y|^{\alpha}} 
\]
is finite. The corresponding \textbf{H\"{o}lder norm} is defined by
\[
\|u\|_{C^{0,\alpha}(\overline{\Omega } )} := \|u\|_{L^{\infty}(\Omega)} + [u]_{C^{0,\alpha}(\overline{\Omega } )}.
\]
Moreover, when $\alpha=1$, the space $C^{0,1}(\overline{\Omega } )$ coincides with the classical space of \textbf{Lipschitz continuous functions}.

More generally, for any integer $k \in \mathbb{N}$ and $\alpha \in (0, 1]$, the space $C^{k,\alpha}(\overline{\Omega } )$ is defined as the set of functions $u \in C^k(\overline{\Omega } )$ such that the norm
\[
\|u\|_{C^{k,\alpha}(\overline{\Omega } )} := \|u\|_{C^k(\overline{\Omega } )} + [D^k u]_{C^{0,\alpha}(\overline{\Omega } )}
\]
is finite, where
\[
\|u\|_{C^k(\overline{\Omega } )} := \|u\|_{L^{\infty}(\overline{\Omega })} + \sum_{j=1}^{k} \sum_{l \in \mathbb{N}^n \atop{|l|=j}} \|D^l u\|_{L^{\infty}(\overline{\Omega })} \quad \text{and} \quad [D^k u]_{C^{0,\alpha}(\overline{\Omega })} \defeq \sum_{|l|=k} [\partial^{l} u]_{C^{0, \alpha}(\overline{\Omega})} 
\]

\end{definition}

\bigskip

There exist several equivalent definitions and characterizations of H\"{o}lder spaces. The following property will be particularly instrumental for our purposes:

\medskip

\noindent
\textbf{(Property P1)} Let $\Omega \subset \mathbb{R}^n$, $k \in \mathbb{N}$, and $\alpha \in (0,1]$. Suppose that $(u_i)_{i \in \mathbb{N}}$ is a sequence of functions satisfying
\[
\|u_i\|_{C^{k,\alpha}(\overline{\Omega})} \leq \mathrm{C}_0,
\]
for some constant $\mathrm{C}_0 > 0$ independent of $i$, and that $u_i \to u_0$ uniformly in $\overline{\Omega}$. Then,
\[
u_0 \in C^{k,\alpha}(\overline{\Omega}) \quad \text{and} \quad \|u_0\|_{C^{k,\alpha}(\overline{\Omega})} \leq \mathrm{C}_0.
\]

\bigskip

\begin{theorem}[{\bf Arzelà--Ascoli {\cite[Theorem 1.7]{FR-O}}}] \label{Theorem-Ascoli}
Let $\Omega \subset \mathbb{R}^n$ and $\alpha \in (0, 1)$. Suppose that $(u_i)_{i \in \mathbb{N}}$ is a sequence of functions satisfying
\[
\|u_i\|_{C^{0,\alpha}(\overline{\Omega})} \leq \mathrm{C}_0.
\]
Then, there exists a subsequence $(u_{i_j})_{j \in \mathbb{N}}$ that converges uniformly to a function $u \in C^{0,\alpha}(\overline{\Omega})$.

More generally, combining this result with  \textnormal{\textbf{Property P1}}, we obtain the following: if
\[
\|u_i\|_{C^{k,\alpha}(\overline{\Omega})} \leq \mathrm{C}_0,
\]
for some $\alpha \in (0, 1)$ and integer $k \in \mathbb{N}$, then there exists a subsequence $(u_{i_j})_{j\in \mathbb{N}}$ that converges in the $C^k(\overline{\Omega})$ norm to a function $u \in C^{k,\alpha}(\overline{\Omega})$.
\end{theorem}

\bigskip
\textbf{Interpolation Inequalities in H\"{o}lder Spaces --- {\cite[Lemma 6.35]{Tru01}}.} 
A fundamental tool employed throughout this article is the following interpolation inequality.

\textit{For any $0 \leq \gamma < \alpha < \beta \leq 1$ and every $\varepsilon > 0$, we have
\begin{equation} \label{eq:interpolation1}
\|u\|_{C^{0,\alpha}(\Omega)} \leq \mathrm{C}_\varepsilon \|u\|_{C^{0,\gamma}(\Omega)} + \varepsilon \|u\|_{C^{0,\beta}(\Omega)},
\end{equation}
where $\mathrm{C}_\varepsilon$ is a constant depending only on $n$ and $\varepsilon$. (When $\gamma = 0$, the space $C^{0,\gamma}$ is interpreted as $L^\infty$.)}

This inequality is a consequence of the \textbf{standard interpolation estimate}:
\[
\|u\|_{C^{0,\alpha}(\Omega)} \leq \|u\|_{C^{0,\gamma}(\Omega)}^{t} \|u\|_{C^{0,\beta}(\Omega)}^{1 - t}, \quad \text{\textit{with}} \quad t = \frac{\beta - \alpha}{\beta - \gamma}.
\]

More generally, inequality \eqref{eq:interpolation1} can be extended to higher-order H\"{o}lder norms. In particular, we will frequently utilize the following special cases: \textit{for any $\varepsilon > 0$ and $\alpha \in (0,1)$,
\begin{equation}\label{eq:interpolation-gradiente}
    \|D u\|_{L^\infty(\Omega)} \leq \mathrm{C}_\varepsilon \|u\|_{L^\infty(\Omega)} + \varepsilon [D u]_{C^{0,\alpha}(\Omega)},
\end{equation}
and
\begin{equation} \label{eq:interpolation2}
\|u\|_{C^2(\Omega)} = \|u\|_{C^{1,1}(\Omega)} \leq \mathrm{C}_\varepsilon \|u\|_{L^\infty(\Omega)} + \varepsilon [D^2 u]_{C^{0,\alpha}(\Omega)}.
\end{equation}}

\vspace{0.5cm}
The following result represents a refinement of the classical \textbf{Liouville’s theorem}, characterizing the structure of harmonic functions in $\mathbb{R}^n$ subject to polynomial growth conditions. (see, \cite[Proposition 1.19]{FR-O}).

\begin{theorem}[{\bf Liouville’s theorem with growth}] \label{Theorem-Liouville's}
Assume that $u$ is a solution of $\Delta u = 0$ in $\mathbb{R}^n$ satisfying
\[
|u(x)| \leq \mathrm{C}(1 + |x|^\gamma) \quad \text{for all } x \in \mathbb{R}^n,
\]
for some constant $\mathrm{C} > 0$ and exponent $\gamma > 0$. Then, $u$ is a polynomial of degree at most $\lfloor \gamma \rfloor$.

\end{theorem}

\vspace{0.3cm}

The following abstract lemma serves as a fundamental tool in establishing the transition from Lemma~\ref{Interpolationinterior} to Theorems~\ref{thm:schauder} and~\ref{T2}.

\begin{lemma}[{\cite[Lemma 2.27]{FR-O}}]\label{lemma-abstrato}

Let $\kappa \in \mathbb{R}_{+}$ and $\mathrm{c}_0 > 0$. Let $\mathcal{S}$ be a non-negative function defined on the class of open convex subsets of $B_1,$ and suppose that $\mathcal{S}$ is subadditive. That is, if $\mathscr{A}, \mathscr{A}_1, \mathscr{A}_2, \ldots, \mathscr{A}_{\mathrm{N}_0}$ are open convex subsets of $B_1$ with $ \displaystyle \mathscr{A} \subset \bigcup_{j=1}^{N} \mathscr{A}_j$, then
\begin{align*}
    \mathcal{S}(\mathscr{A}) \leq \sum_{j=1}^{\mathrm{N}_0} \mathcal{S}(\mathscr{A}_j).
\end{align*}
Then, there exists a small constant $\delta = \delta(n,\kappa) > 0$ such that if
\begin{align*}
    \rho^{\kappa} \mathcal{S}\left(B_{\frac{\rho}{2}}\left(x_0\right)\right) \leq \delta \rho^{\kappa} \mathcal{S}\left(B_\rho(x_0)\right) + \mathrm{c}_0, \hspace{0.3cm} \forall \,\,\,B_\rho(x_0) \subset B_1,
\end{align*}
it follows that $ \displaystyle \mathcal{S}\left(B_{\frac{1}{2}}\right) \leq \mathrm{C}(n,\kappa)\mathrm{c}_0$.

\end{lemma}

\medskip

Finally, recall that viscosity solutions satisfy the following stability property (see, for instance, \cite[Proposition 4.11]{CafCabre1995}) and  \cite[Theorem 3.8]{CCKS}.

\begin{proposition}\label{Stability-Prop}
Let $(F_k)_{k \in \mathbb{N}}$ be a sequence of fully nonlinear elliptic operators with ellipticity constants $0< \lambda \leq \Lambda$. Let $(u_k)_{k \in \mathbb{N}} \subset C^0(\Omega)$ be viscosity solutions to
\[
F_k(D^2 u_k, x) +\langle \mathfrak{B}_k(x),Du_k\rangle = f_k(x) \quad \text{in } \Omega,
\]
where $(f_k)_{k \in \mathbb{N}}$ is a sequence of continuous functions and $(\mathfrak{B}_k)_{k \in \mathbb{N}}$ is a sequence of continuous vector fields. Suppose further that $F_k \to F$ locally uniformly in $\text{Sym}(n)$, and that $u_k\to u$, $\mathfrak{B}_k\to \mathfrak{B}$ and $f_k \to f$ locally uniformly in $\Omega$. Then,
\[
F(D^2 u, x) +\langle \mathfrak{B}(x),Du\rangle = f(x) \quad \text{in } \Omega
\]
in the viscosity sense.
\end{proposition}

\section{Schauder a priori estimates}

We now present the proofs of Theorems \ref{T1} and \ref{T2}, in that order.

\subsection{Interior estimates}

To prove Theorem \ref{T1}, we require an auxiliary lemma that establishes an interpolation inequality involving the Hölder seminorm of the Hessian $D^{2}u$ and the right-hand side of estimate \eqref{estinterior}, where the norm $\|u\|_{L^{\infty}(\mathrm{B}_{1})}$ is replaced by $\|D^{2}u\|_{L^{\infty}(\mathrm{B}_{1})}$.

\begin{lemma}\label{Interpolationinterior}
Suppose the assumptions of Theorem \ref{T1} hold. Then, given $\varepsilon > 0$, there exists a constant $\mathrm{C}_{\varepsilon} > 0$, depending only on $\varepsilon$, $n$, $\lambda$, $\Lambda$, $\alpha$, and $\omega$, such that

\begin{equation}\label{proofeqinterior1.0}
 [D^{2}u]_{\alpha,\mathrm{B}_{1/2}} \leq \varepsilon [D^{2}u]_{\alpha,\mathrm{B}_{1}} + \mathrm{C}_{\varepsilon} \left( \|u\|_{L^{\infty}(\mathrm{B}_{1})} + \|f\|_{C^{0,\alpha}(\mathrm{B}_{1})} + \|\mathfrak{B}\|_{C^{0,\alpha}(\mathrm{B}_{1};\mathbb{R}^{n})} \right).
\end{equation}

\end{lemma}

\begin{proof}

Observe that, in order to derive \eqref{proofeqinterior1.0}, it suffices to establish the following inequality: for $\varepsilon \ll 1$, one has
\begin{equation}\label{proofeqinterior1}
 [D^{2}u]_{\alpha,\mathrm{B}_{1/2}} \leq \varepsilon [D^{2}u]_{\alpha,\mathrm{B}_{1}} + \mathrm{C}_{\varepsilon} \left( \|D^{2}u\|_{L^{\infty}(\mathrm{B}_{1})} + \|f\|_{C^{0,\alpha}(\mathrm{B}_{1})} + \|\mathfrak{B}\|_{C^{0,\alpha}(\mathrm{B}_{1};\mathbb{R}^{n})} \right),
\end{equation}
for some constant $\mathrm{C}_{\varepsilon}>0$ depending on $\varepsilon>0$ and universal parameters.

\medskip
Indeed, by the interpolation inequality \ref{eq:interpolation2}, if $\delta = \varepsilon/\mathrm{C}_{\varepsilon} $, then
\begin{equation*}
    \|D^{2}u\|_{L^{\infty}(\mathrm{B}_{1})} \leq \delta [D^{2}u]_{\alpha,\mathrm{B}_{1}} + \mathrm{C}_{\delta} \|u\|_{L^{\infty}(\mathrm{B}_{1})}.
\end{equation*}

Therefore,
\begin{align*}
    [D^{2}u]_{\alpha,\mathrm{B}_{1/2}} &\leq \varepsilon [D^{2}u]_{\alpha,\mathrm{B}_{1}} + \mathrm{C}_{\varepsilon} \left( \|D^{2}u\|_{L^{\infty}(\mathrm{B}_{1})} + \|f\|_{C^{0,\alpha}(\mathrm{B}_{1})} + \|\mathfrak{B}\|_{C^{0,\alpha}(\mathrm{B}_{1};\mathbb{R}^{n})} \right) \\
    &\leq \varepsilon [D^{2}u]_{\alpha,\mathrm{B}_{1}} + \mathrm{C}_{\varepsilon} \left( \frac{\varepsilon}{\mathrm{C}_{\varepsilon}} [D^{2}u]_{\alpha,\mathrm{B}_{1}} + \mathrm{C}_{\delta} \|u\|_{L^{\infty}(\mathrm{B}_{1})} + \|f\|_{C^{0,\alpha}(\mathrm{B}_{1})} + \|\mathfrak{B}\|_{C^{0,\alpha}(\mathrm{B}_{1};\mathbb{R}^{n})} \right) \\
    &\leq 2\varepsilon [D^{2}u]_{\alpha,\mathrm{B}_{1}} + \max\left\{ \mathrm{C}_{\delta}\mathrm{C}_{\varepsilon},\hspace{0.1cm} 1 \right\} \left( \|u\|_{L^{\infty}(\mathrm{B}_{1})} + \|f\|_{C^{0,\alpha}(\mathrm{B}_{1})} + \|\mathfrak{B}\|_{C^{0,\alpha}(\mathrm{B}_{1};\mathbb{R}^{n})} \right)\\
    &\leq \widehat{\varepsilon} [D^{2}u]_{\alpha,\mathrm{B}_{1}} + \widehat{\mathrm{C}}_{\widehat{\varepsilon}}\left( \|u\|_{L^{\infty}(\mathrm{B}_{1})} + \|f\|_{C^{0,\alpha}(\mathrm{B}_{1})} + \|\mathfrak{B}\|_{C^{0,\alpha}(\mathrm{B}_{1};\mathbb{R}^{n})} \right).
\end{align*}

\bigskip

We establish inequality \eqref{proofeqinterior1} via a \textit{reductio ad absurdum} argument. Suppose, for contradiction, that there exist $\varepsilon_{0}>0$ and functions $u_{j} \in C^{2,\alpha}(B_{1})$, $f_{j} \in C^{0,\alpha}(B_{1})$, and $F_{j}\in \mathscr{F}$ satisfying
\begin{equation}\label{proofeqinterior2}
F_{j}(D^{2}u_{j}, x) + \langle \mathfrak{B}_{j}(x), Du_{j} \rangle = f_{j}(x) \quad \text{in } B_{1},
\end{equation}
with $\|u_{j}\|_{C^{2,\alpha}(B_{1})} \leq C_{0}$ for all $j\in\mathbb{N}$, yet
\begin{equation}\label{proofeqinterior3}
[D^{2}u_{j}]_{\alpha,B_{1/2}} > \varepsilon_0[D^{2}u_{j}]_{\alpha,B_{1}} + j(\|D^{2}u_{j}\|_{L^{\infty}(B_{1})} + \|f_{j}\|_{C^{0,\alpha}(B_{1})} + \|\mathfrak{B}_{j}\|_{C^{0,\alpha}(B_{1};\mathbb{R}^{n})}).   
\end{equation}

This inequality implies $[D^{2}u_{j}]_{\alpha,B_{1/2}} > 0$. Consequently, by definition of the supremum, for each $j\in\mathbb{N}$ there exist distinct points $x_{j},y_{j}\in B_{1/2}$ such that
\begin{equation}\label{proofeqinterior4}
\frac{1}{2}[D^{2}u_{j}]_{\alpha,B_{1/2}} \leq \frac{\|D^{2}u_{j}(x_{j}) - D^{2}u_{j}(y_{j})\|}{|x_{j} - y_{j}|^{\alpha}}.
\end{equation}

Defining $\rho_{j} = |x_{j} - y_{j}|$, we deduce from \eqref{proofeqinterior3} and \eqref{proofeqinterior4} that $0 < \rho_{j} < (4/j)^{1/\alpha} \to 0$.

Consider the rescaled functions $v_j$ defined on $B_{1/(2\rho_{j})}$ by
\begin{equation*}
v_{j}(x) = \frac{u_{j}(x_{j} + \rho_{j}x) - P_{j}(x)}{\rho_{j}^{2+\alpha}[D^{2}u_{j}]_{\alpha,B_{1}}},
\end{equation*}
where $P_{j}$ is the second-order Taylor polynomial of $u_j$ at $x_j$:
\begin{equation*}
P_{j}(x) = u_{j}(x_{j}) + \rho_{j}\langle Du_{j}(x_{j}), x \rangle + \frac{\rho_{j}^{2}}{2}x^{t}D^{2}u_{j}(x_j)x.
\end{equation*}

These functions satisfy
\[
v_{j}(0) = |Dv_{j}(0)| = \|D^{2}v_{j}(0)\| = 0 \quad \text{and} \quad [D^{2}v_{j}]_{\alpha,B_{1/(2\rho_{j})}} \leq 1 \quad \forall j\in\mathbb{N}.
\]

For $z_{j} = \frac{x_{j}-y_{j}}{\rho_{j}} \in \partial B_{1}$, we obtain from \eqref{proofeqinterior3} and \eqref{proofeqinterior4}:
\begin{equation}\label{proofeqinterior5}
\|D^{2}v_{j}(z_{j})\| = \frac{\|D^{2}u_{j}(y_{j}) - D^{2}u_{j}(x_{j})\|}{\theta_{j}} \geq \frac{1}{2}\frac{[D^{2}u_{j}]_{\alpha,B_{1/2}}}{[D^{2}u_{j}]_{\alpha,B_{1}}} > \frac{\varepsilon_{0}}{2},
\end{equation}
where $\theta_{j} = \rho_{j}^{\alpha}[D^{2}u_{j}]_{\alpha,B_{1}}$.

Since $u_{j}$ solves \eqref{proofeqinterior2}, the rescaled function $v_{j}$ satisfies
\[
\overline{F}_{j}(D^{2}v_{j},x) + \langle \overline{\mathfrak{B}}_{j}(x), Dv_{j} \rangle = \overline{f}_{j}(x) \quad \text{in } B_{\frac{1}{2\rho_{j}}},
\]
with the rescaled operator, vector field, and source term given by:
$$
\left\{
\begin{array}{rcl}
\overline{F}_{j}(X,x) &=& \theta_{j}^{-1}\left[F_{j}(\theta_{j}X + D^{2}u_{j}(x_{j}), x_{j} + \rho_{j}x) - F_{j}(D^{2}u_{j}(x_{j}), x_{j} + \rho_{j}x)\right], \\
\overline{\mathfrak{B}}_{j}(x) &= & \rho_{j}\mathfrak{B}_{j}(x_{j} + \rho_{j}x), \\
\overline{f}_{j}(x) &=& \theta_{j}^{-1}\left[(f_{j}(x_{j} + \rho_{j}x) - f_{j}(x_{j})) - (F_{j}(D^{2}u_{j}(x_{j}), x_{j} + \rho_{j}x) - F_{j}(D^{2}u_{j}(x_{j}), x_{j}))\right] \\
& + &\theta_{j}^{-1}\langle \mathfrak{B}_{j}(x_{j} + \rho_{j}x) - \mathfrak{B}_{j}(x_{j}), Du_{j}(x_{j}) \rangle \\
& - &\theta_{j}^{-1}\rho_{j}\langle \mathfrak{B}_{j}(x_{j} + \rho_{j}x), D^{2}u_{j}(x_{j})x \rangle.
\end{array}
\right.
$$

From the structural conditions \(\mathrm{(A1)}\)–\(\mathrm{(A3)}\), it follows—up to extraction of a subsequence—that \(F_j \to F_0\) and \(DF_j \to DF_0\) locally uniformly in \(\mathrm{Sym}(n)\), for some \(F_0 \in \mathscr{F}\). By the uniform boundedness of the sequence \((u_j)_{j \in \mathbb{N}}\) in the \(C^{2,\alpha}\)-norm and the compactness of \(\partial \mathrm{B}_1\), we obtain—again, up to a subsequence—that \(u_j \to u_0\) locally uniformly in \(\mathrm{B}_1\), and \(z_j \to z_0 \in \partial \mathrm{B}_1\), respectively. For the same reasons, \(v_j \to v_0\) in the \(C^2\)-norm on compact subsets of \(\mathbb{R}^n\), for some function \(v_0\) of class \(C^{2,\alpha}\). Based on these observations, we deduce that \(v_0\) satisfies
\begin{equation}\label{proofeqinterior6}
v_0(0) = |Dv_0(0)| = \|D^2v_0(0)\| = 0, \quad \text{and} \quad [D^2v_0]_{\alpha,\mathbb{R}^n} \leq 1,
\end{equation}
and
\begin{equation}\label{proofeqinterior7}
\|D^2v_0(z_0)\| > \frac{\varepsilon_0}{2}.
\end{equation}

On the other hand, for a fixed \(r > 0\), there exists \(j_0 \in \mathbb{N}\) such that \(\mathrm{B}_r \subset \mathrm{B}_{1/(2\rho_j)}\) for all \(j \geq j_0\) (since \(\rho_j \to 0\)). Hence, in \(\mathrm{B}_r\), for all \(j \geq j_0\), we have:
\begin{align*}
|\overline{f}_j(x)| &= \theta_j^{-1}|f_j(x_j + \rho_j x) - f_j(x_j)| + \mathrm{c}\theta_j^{-1}\|D^2u_j(x_j)\|\rho_j^\alpha |x| \\
&\quad + \theta_j^{-1}\|\mathfrak{B}_j(x_j + \rho_j x) - \mathfrak{B}_j(x_j)\|\cdot |Du_j(x_j)| \\
&\leq \mathrm{c}\theta_j^{-1} \rho_j^\alpha r^\alpha\left([f_j]_{\alpha,\mathrm{B}_1} + \|D^2u_j\|_{L^\infty(\mathrm{B}_1)} + [\mathfrak{B}_j]_{\alpha,\mathrm{B}_1} \|Du_j\|_{L^\infty(\mathrm{B}_1)}\right) \\
&\leq \mathrm{c}([D^2u_j]_{\alpha,\mathrm{B}_1})^{-1} r^\alpha \left( \frac{[D^2u_j]_{\alpha,\mathrm{B}_{1/2}}(1 + \|Du_j\|_{L^\infty(\mathrm{B}_1)})}{j} \right) \\
&\leq \mathrm{c} \frac{r^\alpha(1 + \|Du_j\|_{L^\infty(\mathrm{B}_1)})}{j},
\end{align*}
where \(\mathrm{c} > 0\) is a constant such that \(\tau(t) \leq \mathrm{c}t^\alpha\). This implies that \(\|\overline{f}_j\|_{L^\infty(\mathrm{B}_r)} \to 0\). Moreover, for all \(\mathrm{X} \in \mathrm{Sym}(n)\),
\begin{align*}
|\overline{F}_j(\mathrm{X},x) - \overline{F}_j(\mathrm{X},y)| &\leq \mathrm{C}' \theta_j^{-1} [\theta_j\|\mathrm{X}\| + 2\|D^2u_j\|_{L^\infty(\mathrm{B}_1)}] \rho_j^\alpha |x - y|^\alpha \\
&\leq \mathrm{C}' \left( \|\mathrm{X}\| \rho_j + \frac{2}{j} \right) |x - y|^\alpha \to 0.
\end{align*}
Thus, \(\overline{F}_j\) converges to an operator with constant coefficients—specifically, \(DF_0(\mathrm{A}_0,x_0)\), where \(x_0 = \lim_{j \to \infty} x_j\). Indeed, given any \(\phi \in C^2(\mathrm{B}_r)\), for sufficiently large \(j\), it follows that
\begin{eqnarray}
|\overline{F}_{j}(D^{2}\phi,x)+\langle \overline{\mathfrak{B}}_{j}(x),D\phi\rangle-\overline{f}_{j}-DF_{0}(\mathrm{A}_{0},x_{0})(D^{2}\phi)|&\leq& |\overline{f}_{j}(x)|+\rho_{j}\|\mathfrak{B}_{j}\|_{L^{\infty}(\mathrm{B}_{1};\mathbb{R}^{n})}|D\phi|\nonumber\\
&+&\Bigg|\int_{0}^{1}DF_{j}(t\theta_{j}D^{2}\phi+D^{2}u_{j}(x_{j}),x_{j}+\rho_{j}x)(D^{2}\phi) dt\nonumber\\
&-&DF_{j}(D^{2}u_{j}(x_{j}),x_{j}+\rho_{j}x)(D^{2}\phi)\Bigg|\nonumber\\
&+&|DF_{j}(D^{2}u_{j}(x_{j}),x_{j}+\rho_{j}x)(D^{2}\phi)\nonumber\\
&-&DF_{0}(\mathrm{A}_{0},x_{0})(D^{2}\phi)|\nonumber\\
&\leq&\frac{r^{\alpha}}{j}\rho_{j}\frac{[D^{2}u_{j}]_{\alpha,\mathrm{B}_{1}}}{j}|D\phi|+  \omega(\theta_{j}\|D^{2}\phi\|)\|D^{2}\phi\|\nonumber\\
&+&|DF_{j}(D^{2}u_{j}(x_{j}),x_{j}+\rho_{j}x)(D^{2}\phi)\nonumber\\
&-&DF_{0}(\mathrm{A}_{0},x_{0})(D^{2}\phi)|
\end{eqnarray}

Consequently, from the previous estimate and the uniform convergence of \((F_{j})\) and \((DF_{j})\), it follows that 
\[
\lim_{j\to\infty}\left\|\overline{F}_{j}(D^{2}\phi,x)+\langle \overline{\mathfrak{B}}_{j}(x),D\phi\rangle-\overline{f}_{j}-DF_{0}(\mathrm{A}_{0},x_{0})(D^{2}\phi)\right\|_{L^{\infty}(\mathrm{B}_{r})}=0.
\]
Hence, by the stability result (Proposition \ref{Stability-Prop}), \(v_{0}\) is a solution to 
\begin{equation}
\mathcal{L}_0 v_0 = DF_{0}(\mathrm{A}_{0},x_{0})(D^2v_{0})=0 \quad \text{in} \quad \mathbb{R}^{n}.
\end{equation}
Without loss of generality, up to a rotation, we may assume that \(v_{0}\) is a harmonic function in \(\mathbb{R}^{n}\). Consequently, \(v_{0}\) is smooth, and all its partial derivatives are harmonic functions as well. From the condition \([D^{2}v_{0}]_{\alpha,\mathbb{R}^{n}}\leq 1\), we deduce that \(D^{2}v_{0}\) exhibits sublinear growth at infinity. Therefore, by Liouville's Theorem \ref{Theorem-Liouville's}, \(D^{2}v_{0}\) must be constant. Since \(D^{2}v_{0}(0)=0\) (as noted in \eqref{proofeqinterior6}), we conclude that \(D^{2}v_{0} \equiv 0\), which contradicts \eqref{proofeqinterior7}. This completes the proof. 
\end{proof}

\subsection{Proof of Theorem \ref{T1}}\label{Sec-ProofThe01}

Finally, we are now in a position to present the proof of Theorem~\ref{T1}.

\begin{proof}
Define $u_\rho: \mathrm{B}_1 \to \mathbb{R}$ by $u_\rho(x) := u(\rho x + x_0)$, where $x_0 \in \mathrm{B}_{1/2}$ and $\rho \in (0, \frac{1}{2}]$. Then, $u_\rho$ satisfies
\begin{equation*}
    F_\rho(D^2 u_\rho, x) + \langle \mathfrak{B}_\rho(x), D u_\rho \rangle = f_\rho(x) \quad \text{in} \ \mathrm{B}_1,
\end{equation*}
where
\begin{equation*}
    F_\rho(\mathrm{M}, x) := \rho^2 F(\rho^{-2} \mathrm{M}, \rho x + x_0), \quad
    \mathfrak{B}_\rho(x) := \rho \mathfrak{B}(\rho x + x_0), \quad
    f_\rho(x) := \rho^2 f(\rho x + x_0),
\end{equation*}
verify the assumptions $(\mathbf{A}1)-(\mathbf{A}3)$, and the corresponding H\"{o}lder seminorms satisfy
\begin{equation*}
    \left[ D^2 u_\rho \right]_{\alpha, \mathrm{B}_{1/2}} = \rho^{2+\alpha} \left[ D^2 u \right]_{\alpha,  \mathrm{B}_{\rho / 2}(x_0)}, \quad 
    \left[ \mathfrak{B}_\rho \right]_{\alpha, \mathrm{B}_1} = \rho^{1+\alpha} \left[ \mathfrak{B} \right]_{\alpha, \mathrm{B}_\rho(x_0)}, \quad \text{and} \quad 
    \left[ f_\rho \right]_{\alpha, \mathrm{B}_1} = \rho^{2+\alpha} \left[ f \right]_{\alpha, \mathrm{B}_\rho(x_0)}.
\end{equation*}

Now, applying the estimate from Lemma~\ref{Interpolationinterior} with a fixed $\delta > 0$, we obtain
\begin{align*}
\rho^{2+\alpha} \left[ D^2 u \right]_{\alpha, \mathrm{B}_{\rho/2}(x_0)} &= \left[ D^2 u_\rho \right]_{\alpha, \mathrm{B}_{1/2}} \\
&\leq \delta \left[ D^2 u_\rho \right]_{\alpha, \mathrm{B}_1} + \mathrm{C}_\delta \left( \|  u_\rho \|_{L^\infty(\mathrm{B}_1)} + \| f_\rho \|_{C^{0,\alpha}(\mathrm{B}_1)} + \| \mathfrak{B}_\rho \|_{C^{0,\alpha}(\mathrm{B}_1; \mathbb{R}^n)} \right) \\
&= \delta \rho^{2+\alpha} \left[ D^2 u \right]_{\alpha, \mathrm{B}_\rho(x_0)} + \mathrm{C}_\delta \Big(  \| u \|_{L^\infty(\mathrm{B}_\rho(x_0))} + \rho^2 \| f \|_{L^\infty(\mathrm{B}_\rho(x_0))} \\
&\quad + \rho \| \mathfrak{B} \|_{L^\infty(\mathrm{B}_\rho(x_0); \mathbb{R}^n)} + \rho^{1+\alpha} \left[ \mathfrak{B} \right]_{\alpha, \mathrm{B}_\rho(x_0)} + \rho^{2+\alpha} \left[ f \right]_{\alpha, \mathrm{B}_\rho(x_0)} \Big) \\
&\leq \delta \rho^{2+\alpha} \left[ D^2 u \right]_{\alpha, \mathrm{B}_\rho(x_0)} + \mathrm{C}_\delta \left( \|  u \|_{L^\infty(\mathrm{B}_1)} + \| f \|_{C^{0,\alpha}(\mathrm{B}_1)} + \| \mathfrak{B} \|_{C^{0,\alpha}(\mathrm{B}_1; \mathbb{R}^n)} \right).
\end{align*}

Therefore,
\begin{equation*}
\rho^{2+\alpha} \left[ D^2 u \right]_{\alpha, \mathrm{B}_{\rho/2}(x_0)} \leq \delta \rho^{2+\alpha} \left[ D^2 u \right]_{\alpha, \mathrm{B}_\rho(x_0)} + \mathrm{C}_\delta \left( \| u \|_{L^\infty(\mathrm{B}_1)} + \| f \|_{C^{0,\alpha}(\mathrm{B}_1)} + \| \mathfrak{B} \|_{C^{0,\alpha}(\mathrm{B}_1; \mathbb{R}^n)} \right).
\end{equation*}

Finally, invoking Lemma~\ref{lemma-abstrato} with
\[
\kappa := \alpha + 2, \quad
\mathrm{c}_0 := \mathrm{C}_\delta \left( \|  u \|_{L^\infty(\mathrm{B}_1)} + \| f \|_{C^{0,\alpha}(\mathrm{B}_1)} + \| \mathfrak{B} \|_{C^{0,\alpha}(\mathrm{B}_1; \mathbb{R}^n)} \right),
\]
and defining the functional $\mathcal{S}(\mathscr{A}) := \left[ D^2 u \right]_{\alpha, \mathscr{A}}$, which is subadditive on convex open subsets, we conclude from the preceding inequality that
\begin{equation}
    \left[ D^2 u \right]_{\alpha, \mathrm{B}_{1/2}} \leq \mathrm{C}(n,\kappa) \left( \| u \|_{L^\infty(\mathrm{B}_1)} + \| f \|_{C^{0,\alpha}(\mathrm{B}_1)} + \| \mathfrak{B} \|_{C^{0,\alpha}(\mathrm{B}_1; \mathbb{R}^n)} \right).
\end{equation}

In conclusion, invoking the interpolation inequality \ref{eq:interpolation2}, we obtain
\begin{align*}
    \| u \|_{\mathrm{C}^{2, \alpha}(\mathrm{B}_{1/2})} &= \| u \|_{\mathrm{C}^{2} (\mathrm{B}_{1/2})} + [D^2 u]_{C^{0,\alpha}(\mathrm{B}_{1/2})} \\
    &\leq \mathrm{C}_\varepsilon \|u\|_{L^\infty(\mathrm{B}_{1/2})} + \varepsilon [D^2 u]_{C^{0,\alpha}(\mathrm{B}_{1/2})} + [D^2 u]_{C^{0,\alpha}(\mathrm{B}_{1/2})} \\
    &\leq \mathrm{C}_\varepsilon \|u\|_{L^\infty(\mathrm{B}_{1})} + (\varepsilon+1) \left( \mathrm{C}(n,\kappa) \left( \| u \|_{L^\infty(\mathrm{B}_1)} + \| f \|_{C^{0,\alpha}(\mathrm{B}_1)} + \| \mathfrak{B} \|_{C^{0,\alpha}(\mathrm{B}_1; \mathbb{R}^n)} \right) \right) \\
    &\leq \widehat{\mathrm{C}} \left( \| u \|_{L^\infty(\mathrm{B}_1)} + \| f \|_{C^{0,\alpha}(\mathrm{B}_1)} + \| \mathfrak{B} \|_{C^{0,\alpha}(\mathrm{B}_1; \mathbb{R}^n)} \right),
\end{align*}
where 
\[
\widehat{\mathrm{C}} = \max \left\{ \mathrm{C}_{\varepsilon} + (\varepsilon+1)\mathrm{C}_{\delta},\hspace{0.1cm} (\varepsilon+1)\mathrm{C}_{\delta} \right\} = \mathrm{C}_{\varepsilon} + (\varepsilon+1)\mathrm{C}_{\delta}.
\]
\medskip

\noindent This concludes the proof.
\end{proof}

\subsection{Boundary estimates}

Finally, we derive the boundary estimates for the problem under the Dirichlet condition. Our strategy is to prove an auxiliary inequality for solutions to the corresponding Dirichlet problem, and then apply interpolation arguments to establish a proof of Theorem \ref{T2}. For the next result, we assume the compatibility condition
\[
F(D^2g ,x)+\langle \mathfrak{B}(x),Dg\rangle = f(x) \quad \text{in} \,\, \mathrm{T}_{1}
\]
as \( u \in C^{2, \alpha}(B_1^{+}\cup \mathrm{T}_1) \) and \( u = g \) on \( \mathrm{T}_1 \).

\begin{lemma}\label{Interpolationboundary}
Suppose the hypotheses of Theorem \ref{T2} are satisfied. Given \(\varepsilon>0\), there exists a positive constant \(\mathrm{C}^{\ast}_{\varepsilon} = \mathrm{C}^{\ast}_{\varepsilon}(\varepsilon, n, \lambda, \Lambda, \alpha, \omega)\) such that
\begin{equation*}
[D^{2}u]_{\alpha,\mathrm{B}^{+}_{1/2}} \leq \varepsilon [D^{2}u]_{\alpha,\mathrm{B}^{+}_{1}} + \mathrm{C}^{\ast}_{\varepsilon} \left( \|D^{2}u\|_{L^{\infty}(\mathrm{B}_{1}^{+})} + \|f\|_{C^{0,\alpha}(\mathrm{B}_{1}^{+})} + \|\mathfrak{B}\|_{C^{0,\alpha}(\mathrm{B}_{1}^{+};\mathbb{R}^{n})} + \|g\|_{C^{2,\alpha}(\mathrm{T}_{1})} \right).
\end{equation*}
\end{lemma}

\begin{proof}
As in Lemma \ref{Interpolationinterior}, we establish this result by employing a \textit{reductio ad absurdum} argument. Thus, there exist \(\varepsilon_0 > 0\) and sequences of functions \(u_j \in C^{2,\alpha}(\mathrm{B}^{+}_{1} \cup \mathrm{T}_{1})\), \(f_j \in C^{0,\alpha}(\mathrm{B}^{+}_{1} \cup \mathrm{T}_{1})\), \(\mathfrak{B}_j \in C^{0,\alpha}(\mathrm{B}^{+}_{1} \cup \mathrm{T}_{1};\mathbb{R}^{n})\), \(F_{j} \in \mathscr{F}\), and \(g_j \in C^{2,\alpha}(\mathrm{T}_{1})\) such that \(u_{j}\) is a viscosity solution of
\[
\left\{
\begin{array}{rcl}
F_{j}(D^2u_{j},x)+\langle \mathfrak{B}_{j}(x),Du_{j}\rangle &=& f_{j}(x) \quad \text{in} \,\, \mathrm{B}^{+}_{1}\cup\mathrm{T}_{1},\\
u_{j}(x) &=& g_{j}(x)  \quad \text{on} \,\, \mathrm{T}_{1},
\end{array}
\right.
\]
with \(\|u_{j}\|_{C^{2,\alpha}(\mathrm{B}^{+}_{1}\cup\mathrm{T}_{1})} \leq \mathrm{C}_{0}\). Yet,
\begin{equation}\label{proofboundary3}
[D^{2}u_{j}]_{\alpha,\mathrm{B}^{+}_{1/2}} > \varepsilon_{0} [D^{2}u_{j}]_{\alpha,\mathrm{B}^{+}_{1}} + j \left( \|D^{2}u_{j}\|_{L^{\infty}(\mathrm{B}_{1}^{+})} + \|f_{j}\|_{C^{0,\alpha}(\mathrm{B}_{1}^{+})} + \|\mathfrak{B}_{j}\|_{C^{0,\alpha}(\mathrm{B}_{1}^{+};\mathbb{R}^{n})} + \|g_{j}\|_{C^{2,\alpha}(\mathrm{T}_{1})} \right).
\end{equation}

It follows that there exist \(x_{j}, y_{j} \in \mathrm{B}^{+}_{1/2}\), with \(x_{j} \neq y_{j}\), such that
\begin{equation}\label{proofboundary4}
\frac{1}{2} [D^{2}u_{j}]_{\alpha,\mathrm{B}^{+}_{1/2}} \leq \frac{\|D^{2}u_{j}(x_{j}) - D^{2}u_{j}(y_{j})\|}{|x_{j} - y_{j}|^{\alpha}}.
\end{equation}

Defining \(\rho_{j} := |x_{j} - y_{j}|\), it follows from \eqref{proofboundary4} that \(\rho_{j} \to 0\) as \(j \to \infty\).

Next, set \(r_{j} := \frac{\dist(x_{j}, \mathrm{T}_{1})}{\rho_{j}}\). Up to passing to a subsequence, there are two possible scenarios for analysis:
\begin{itemize}
\item[(I)] \(r_{j} \to \infty\) as \(j \to \infty\);
\item[(II)] \(\theta := \sup_{j \in \mathbb{N}} r_{j} < \infty\).
\end{itemize}

In the first case, we consider the rescaled function
\[
v_{j}(x) := \frac{u_{j}(x_{j} + \rho_{j} x) - P_{j}(x)}{\rho_{j}^{2 + \alpha} [D^{2}u_{j}]_{\alpha,\mathrm{B}_{1}^{+}}}, \quad x \in \mathrm{B}_{r_{j}},
\]
where
\[
P_{j}(x) := u_{j}(x_{j}) + \rho_{j} \langle Du_{j}(x_{j}), x \rangle + \frac{\rho_{j}^{2}}{2} x^{\top} D^{2}u_{j}(x_j) x.
\]

The function \(v_j\) is well-defined. Proceeding analogously to the interior case (see Lemma \ref{Interpolationinterior}), we arrive at a contradiction employing a blow-up argument combined with the Liouville-type Theorem \ref{Theorem-Liouville's}.

On the other hand, in case (II), we observe that \(0<\frac{\dist(x_{j},\mathrm{T}_{1})}{2}< \theta \rho_{j}\), so there exists \(t_{j}\in\mathrm{T}_{\frac{1}{2}}\) such that \(|x_{j}-t_{j}|\leq 2\theta\rho_{j}\). Define the auxiliary function
\begin{eqnarray*}
v_{j}(x)=\frac{u_{j}(t_{j}+\rho_{j}x)-\bar{P}_{j}(x)}{\rho_{j}^{2+\alpha}[D^{2}u_{j}]_{\alpha,\mathrm{B}_{1}^{+}}}, \quad x\in\mathrm{B}_{\frac{1}{2\rho_{j}}}\cup \mathrm{T}_{\frac{1}{2\rho_{j}}},
\end{eqnarray*}
where
\begin{eqnarray*}
\bar{P}_{j}(x)=u_{j}(t_{j})+\rho_{j}\langle Du_{j}(t_{j}),x\rangle+\frac{\rho_{j}^{2}}{2}x^{t}D^{2}u_{j}(t_j)x.
\end{eqnarray*}
Analogously to the interior case, \(v_{j}(0)=|Dv_{j}(0)|=\|D^{2}v_{j}(0)\|=0\) and \([D^{2}v_{j}]_{\alpha,\mathrm{B}^{+}_{1}}\leq 1\) for all \(j\in\mathbb{N}\). For sufficiently large \(j\), the points \(\bar{x}_{j}=\frac{x_{j}-t_{j}}{\rho_{j}}\) and \(\bar{y}_{j}=\frac{y_{j}-t_{j}}{\rho_{j}}\) lie in a compact subset of \(\overline{\mathbb{R}^{n}_{+}}\). Up to a subsequence, we may assume that \(\bar{x}_{j}, \bar{y}_{j}\to w_{0}\) as \(j\to \infty\), since \(x_{j}-y_{j}\to 0\). In this case, inequality \eqref{proofboundary4} implies
\begin{eqnarray}\label{proofboundary5}
\frac{\varepsilon_{0}}{2}<\frac{\|D^{2}v_{j}(x_{j})-D^{2}v_{j}(y_{j})\|}{\rho_{j}^{\alpha}[D^{2}u_{j}]_{\alpha,\mathrm{B}^{+}_{1}}}&\leq&\frac{\|D^{2}v_{j}(x_{j})-D^{2}v_{j}(t_{j})\|}{\rho_{j}^{\alpha}[D^{2}u_{j}]_{\alpha,\mathrm{B}^{+}_{1}}}+\frac{\|D^{2}v_{j}(y_{j})-D^{2}v_{j}(t_{j})\|}{\rho_{j}^{\alpha}[D^{2}u_{j}]_{\alpha,\mathrm{B}^{+}_{1}}} \nonumber\\
&=&\|D^{2}v_{j}(\bar{x}_{j})\|+\|D^{2}v_{j}(\bar{y}_{j})\|.
\end{eqnarray}
Furthermore, it can be verified that \(v_{j}\) satisfies
\[
\left\{
\begin{array}{rcl}
\overline{F}_{j}(D^2v_{j},x) +\langle\overline{\mathfrak{B}}_{j}(x),Dv_{j}\rangle&=& \overline{f}_{j}(x) \quad \text{in} \ \mathrm{B}^{+}_{\frac{1}{2\rho_{j}}}\cup\mathrm{T}_{\frac{1}{2\rho_{j}}},\\
v_{j}(x)&=& \overline{g}_{j}(x)  \quad \text{on} \ \mathrm{T}_{\frac{1}{2\rho_{j}}}, 
\end{array}
\right.
\]
where
\[
\left\{
\begin{array}{rcl}
\overline{F}_{j}(\mathrm{X}) & \defeq & \theta_{j}^{-1}[F(\theta_{j} \mathrm{X}+D^{2}g_{j}(t_{j}),t_{j}+\rho_{j}x)-F_{j}(D^{2}g_{j}(t_{j}),x_{j}+\rho_{j}x)], \\
\overline{f}_j(x) & \defeq & \theta_{j}^{-1}[ f_{j}(t_{j}+\rho_{j}x)-f_{j}(t_{j})] - \theta_{j}^{-1}[F_{j}(D^{2}u_{j}(t_{j}),t_{j}+\rho_{j}x)-F_{j}(D^{2}u_{j}(t_{j}),t_{j})] \\
&+& \theta_{j}^{-1}\langle \mathfrak{B}_{j}(t_{j}+\rho_{j}x)-\mathfrak{B}_{j}(t_{j}),Du_{j}(t_{j})\rangle - \theta_{j}^{-1}\rho_{j}\langle\mathfrak{B}_{j}(t_{j}+\rho_{j}x),D^{2}u_{j}(t_{j})x\rangle,\\
\overline{\mathfrak{B}}_{j}(x)&\defeq&\rho_{j}\mathfrak{B}_{j}(x_{j}+\rho_{j}x),\\
\overline{g}_{j}(x) & \defeq & \dfrac{g_{j}(t_{j}+\rho_{j}x)-g_{j}(t_{j})-\rho_{j}\langle Dg_{j}(t_{j}),x\rangle-\frac{\rho_{j}^{2}}{2}x^{t}D^{2}g_{j}(t_{j})x}{\rho_{j}^{2+\alpha}[D^{2}u_{j}]_{\alpha,\mathrm{B}^{+}_{1}}},
\end{array}
\right.
\]
and \(\theta_{j}=\rho_{j}^{\alpha}[D^{2}u_{j}]_{\alpha,\mathrm{B}^{+}_{1}}\). As in the interior case, we conclude that \(\overline{f}_{j}\to 0\) locally uniformly in balls \(\mathrm{B}^{+}_{r} \subset \mathbb{R}^{n}_{+}\).

Moreover, since \(\mathrm{T}_{r}\subset \{x_{n}=0\}\), we have for \(j \gg 1\) that in \(\mathrm{T}_{r}\)
\begin{eqnarray}
|\overline{g}_{j}(x)| \leq \frac{[D^{2}g_{j}]_{\alpha,\mathrm{T}_{1}}}{2[D^{2}u_{j}]_{\alpha,\mathrm{B}^{+}_{1}}}r^{2+\alpha} \leq \frac{r^{2+\alpha}}{2j} \to 0 \quad \text{as} \quad j \to \infty,
\end{eqnarray}
where, in the last inequality, we used estimate \eqref{proofboundary3}. Thus, similarly to the interior case, up to a subsequence, we have \(v_{j} \to v_{0} \in C^{2,\alpha}(\overline{\mathbb{R}^{n}_{+}})\) locally uniformly in \(\mathbb{R}^{n}_{+}\), \(F_{j} \to F_{0}\) in \(\operatorname{Sym}(n)\), \(\overline{f}_{j} \to 0\), and \(\overline{g}_{j} \to 0\). By the stability result (Proposition \ref{Stability-Prop}), it follows that \(v_{0}\) solves, in particular,
\[
\left\{
\begin{array}{rcl}
DF_{0}(\mathrm{A}_{0},t_{0})(D^2v_{0}) &=& 0 \quad \text{in} \quad \mathbb{R}^{n}_{+}, \\
v_{0}(x) &=& 0 \quad \text{on} \quad \partial\mathbb{R}^{n}_{+} = \{x_{n}=0\},
\end{array}
\right.
\]
for some matrix \(\mathrm{A}_{0} \in \operatorname{Sym}(n)\) and some point \(t_{0} \in \mathrm{T}_{1/2}\). Moreover, by the previous observations on the sequence \((v_{j})\), we have
\[
v_{0}(0) = |Dv_{0}(0)| = \|D^{2}v_{0}(0)\| = 0, \quad \text{and} \quad [D^{2}v_{0}]_{\alpha,\mathbb{R}^{n}_{+}} \leq 1.
\]
Furthermore, from \eqref{proofboundary5}, we conclude that
\begin{eqnarray}\label{proofboundary6}
\|D^{2}v_{0}(w_{0})\| \geq \frac{\varepsilon_{0}}{4}.
\end{eqnarray}
As in the final step of the proof of Lemma \ref{Interpolationinterior}, we observe that, up to a rotation, \(v_{0}\) is harmonic in the half-space \(\mathbb{R}^{n}_{+}\). Applying the Schwarz Reflection Principle (see \cite[4.12 Theorem]{ABR-Book-2001}), we can extend \(v_{0}\) to the entire space \(\mathbb{R}^{n}\) via the function \(\bar{v}_{0}\) defined by
\[
\bar{v}_{0}(x',x_{n}) =
\begin{cases}
v_{0}(x',x_{n}), & \text{if } x_{n} \geq 0, \\    
-v_{0}(x',-x_{n}), & \text{if } x_{n} < 0,
\end{cases}
\]
which is harmonic in \(\mathbb{R}^{n}\) and satisfies \([D^{2} \bar{v}_{0}]_{\alpha,\mathbb{R}^{n}} \leq 1\). Hence, by Liouville's Theorem \ref{Theorem-Liouville's}, we deduce that \(D^{2} \bar{v}_{0} = 0\), and in particular, \(D^{2}v_{0} = 0\), which contradicts \eqref{proofboundary6}. This completes the proof. 
\end{proof}

In conclusion, the proof of Theorem~\ref{T2} proceeds along the same lines as that of Theorem~\ref{T1} (see Subsection \ref{Sec-ProofThe01}). Hence, to avoid unnecessary repetition, we omit the details.

\subsection*{Final comments and related results}

Motivated by the study of the fully nonlinear Alt–Phillips equation, Wu and Yu in \cite{WH2022} assumed, in addition to the ellipticity condition $(\mathrm{A1})$, that the function \( F: \text{Sym}(n) \to \mathbb{R} \) satisfies the following conditions:
\begin{equation}\label{Condi-WuYu}
\left\{
\begin{array}{l}
F \text{ is convex,} \\
F(\mathbf{O}_n) = 0, \\
\text{the trace operator is a subdifferential of } F \text{ at } \mathbf{O}_n.
\end{array}
\right.
\end{equation}

It is important to note that the convexity assumption ensures the existence of subdifferentials. Given a matrix \( \mathfrak{A} \in \text{Sym}(n) \), a subdifferential of \( F \) at \( \mathfrak{A} \) is a linear operator \( \mathcal{S}_{\mathfrak{A}} : \text{Sym}(n) \to \mathbb{R} \) such that
\[
\mathcal{S}_{\mathfrak{A}}(\mathrm{X}) \leq F(\mathfrak{A} + \mathrm{X}) - F(\mathfrak{A}) \quad \text{for all } \mathrm{X} \in \text{Sym}(n).
\]

Furthermore, observe that the ellipticity condition $(\mathrm{A1})$ implies that \( \mathcal{S}_{\mathfrak{A}} \) constitutes a uniformly elliptic operator.

Additionally, the authors assume that either
\begin{itemize}
\item[\checkmark] $F$ is differentiable at $\mathbf{O}_n$, or
\item[\checkmark] $F$ is $1$-homogeneous, i.e., $F(\mu \mathrm{M}) = \mu F(\mathrm{M})$ for all $\mu > 0$ and $\mathrm{M} \in \text{Sym}(n)$.
\end{itemize}

Under these assumptions, in conjunction with condition \eqref{Condi-WuYu}, it follows that
\[
\mathrm{D}_{\mathrm{M}}F(\mathbf{O}_n) = \mathrm{tr}(\mathrm{M}).
\]

Many studies assume differentiability of the operator, as this condition facilitates the linearization process. Alternatively, some investigations focus on homogeneous operators, wherein appropriately rescaled versions of a solution also satisfy the same equation (cf. \cite{WH2022} and the references therein).

Finally, it is worth emphasizing that the assumptions of convexity and the normalization \( F(\mathbf{O}_n) = 0 \) play a pivotal role in the analysis of free boundary problems involving nonlinear operators (see \cite{WH2022} and the references therein). In particular, these conditions are crucial for deriving Schauder estimates for solutions and align with our present findings.

\subsection*{Acknowledgments}

FAPESP-Brazil has supported J. da Silva Bessa under Grant No. 2023/18447-3. J.V. da Silva has received partial support from CNPq-Brazil under Grant No. 307131/2022-0, Chamada CNPq/MCTI No. 10/2023 - Faixa B - Consolidated Research Groups under Grant No. 420014/2023-3,and by  FAPESP-Brazil under the Grant No.  2025/09344-1-Special Programs-Special Projects-First Projects-Call for Proposals (2025)-1st Cycle. J.V. da Silva and L. Ospina have been supported by FAEPEX-UNICAMP (Project No. 2441/23, Special Calls - PIND - Individual Projects, 03/2023).

\end{document}